\documentclass[a4paper,11pt]{amsart}

\usepackage{amssymb}
\usepackage{fullpage}

\theoremstyle{definition}
\newtheorem{definition}{Definition}
\theoremstyle{plain}
\newtheorem{theorem}[definition]{Theorem}
\newtheorem{proposition}[definition]{Proposition}

\theoremstyle{remark}
\newtheorem{remark}[definition]{Remark}

\DeclareMathOperator{\edim}{edim}
\DeclareMathOperator{\vdim}{vdim}

\def\field{\mathbb{K}}
\def\N{\mathbb{N}}
\def\Z{\mathbb{Z}}
\def\PP{\mathbb{P}}
\def\sys{\mathcal{L}}
\def\rdf{:=}
\def\dff{\it}

\let\to\longrightarrow

\begin{document}

\title{Linear systems in $\mathbb P^3$ with low degrees and low multiplicities}

\author{Marcin Dumnicki}

\dedicatory{
Institute of Mathematics, Jagiellonian University, \\
ul. Grota-Roweckiego 26, 30-348 Krak\'ow, Poland \\
Email address: Marcin.Dumnicki@im.uj.edu.pl \\}


\begin{abstract}
We prove that the linear system of hypersurfaces in $\mathbb P^3$ of degree $d$,
$14 \leq d \leq 40$, with double, triple and quadruple points in general position
are non-special. This solves the cases that have not been completed
in a paper by E. Ballico and M.C. Brambilla.
\end{abstract}

\maketitle

\section{Introduction}

In what follows we assume that the ground field $\field$ is of characteristic zero.
Let $d \in \Z$, let $m_1,\dots,m_r \in \N$.
By $\sys_3(d;m_1,\dots,m_r)$ we denote the linear system of hypersurfaces
(in $\PP^3 \rdf \PP^3(\field)$)
of degree $d$ passing through $r$ points $p_1,\dots,p_r$ in general position with
multiplicities at least $m_1,\dots,m_r$
(the point with multiplicity $m$ will often be called an $m$-point). The dimension of such system is denoted by
$\dim \sys_3(d;m_1,\dots,m_r)$.
Define the {\dff virtual dimension of $L = \sys_3(d;m_1,\dots,m_r)$}
\begin{align*}
\vdim L & \rdf \binom{d+3}{3} - \sum_{j=1}^{r} \binom{m_j+2}{3} - 1 \\
\intertext{and the {\dff expected dimension of $L$}}
\edim L & \rdf \max \{ \vdim L, -1 \}.
\end{align*}
Observe that $\dim L \geq \edim L$. If this inequality is strict then
$L$ is called {\dff special}, {\dff non-special} otherwise.

We will use the following notation: $m^{\times k}$ denotes the
sequence of $m$'s taken $k$ times,
$$m^{\times k} = \underbrace{(m,\dots,m)}_{k}.$$

\section{Results for multiplicities bounded by $4$}

In \cite{BB} then following (see Thm. 1) was proven:

\begin{theorem}
System $L=\sys_3(d;4^{\times x},3^{\times y},2^{\times z})$
is non-special for non-negative integers $d$, $x$, $y$, $z$
satisfying $d \geq 41$.
\end{theorem}

To consider all the lower cases, i.e. $d \leq 40$, authors of \cite{BB}
proposed the algorithm, which essentially is the following:

\medskip\noindent{\bf Algorithm A.}
{\small
\begin{verbatim}
 ----------------------------------------
 d=   -- degree should be chosen
 N=binomial(d+3,3);
 for z from 0 to ceiling(N/4) do
  for y from 0 to ceiling(N/10) do
   for x from 0 to ceiling(N/20) do (
    if ((20*x+10*y+4*z>N-4)and(20*x+10*y+4*z<N+20)) then (
     -- computation of the rank of the interpolation matrix
     -- the rank is maximal if and only if the system is non-special
    )
   )
 ----------------------------------------
\end{verbatim}
}

\medskip

With aid of the program Macaulay2 the following was proven
(\cite{BB}, Thm. 14):

\begin{theorem}
\label{od9do13}
System $L=\sys_3(d;4^{\times x},3^{\times y},2^{\times z})$
is non-special for non-negative integers $d$, $x$, $y$, $z$
satisfying $9 \leq d \leq 13$.
\end{theorem}

In order to prove Theorem \ref{od9do13} one has to check large number of cases,
each case being computation of the rank of the interpolation matrix.
The computation
can be performed over $\mathbb Z_p$ for a small prime $p$, but still
it is time-consuming. The case $d=14$ (and all the next;
for $d=14$ we have $6816$ nearly-square
matrices of size $680$) has not been done in \cite{BB} due to
the length of computations. In this note we show how, using
glueing theorem from \cite{Dum}, one can lower the number of cases.
For example, one need to consider $261$ cases for $d=14$. For greater
$d$'s the difference is even more visible --- originally, for $d=40$,
$2294011$ cases were needed, while in our approach it suffices to deal with $22$ only.

The size of the each matrix is at most (about) $12341$, which is small
enough to be managed by a fast computer.

\section{Main result}
Using a smarter version of Algorithm A we will show the following:

\begin{theorem}
System $L=\sys_3(d;4^{\times x},3^{\times y},2^{\times z})$
is non-special for non-negative integers $d$, $x$, $y$, $z$
satisfying $14 \leq d \leq 40$.
\end{theorem}

We will use the following fact (see Thm. 1 and 9 in \cite{Dum}):

\begin{theorem}
Let $\sys_3(k;\ell_1^{\times s_1},\ell_2^{\times s_2})$ be non-special, let
\begin{align*}
L_1 & = \sys_3(d;m_1,\dots,m_r,\ell_1^{\times s_1},\ell_2^{\times s_2}),\\
L_2 & = \sys_3(d;m_1,\dots,m_r,k+1).
\end{align*}
If either $-1 \leq \vdim L_2 \leq \vdim L_1$ or $\vdim L_1 \leq \vdim L_2 \leq -1$
then in order to show non-specialty of $L_1$ it is enough to show non-specialty
of $L_2$.
\end{theorem}

The above Theorem allows to ``glue'' $s_1$ $\ell_1$-points and
$s_2$ $\ell_2$-points to one $(k+1)$-point
(the glueing will be denoted by $\ell_1^{\times s_1},\ell_2^{\times s_2} \to k+1$)
during computations. We will glue according to the following

\begin{proposition}
Systems $\sys_3(3;2^{\times 5})$ and $\sys_3(9;4^{\times a},3^{\times b})$
for non-negative integers $a$, $b$ satisfying $2a+b=22$ are non-special
of virtual dimension equal to $-1$.
\end{proposition}

\begin{proof}
Non-specialty of $\sys_3(3;2^{\times 5})$ can be done by matrix
computation or, without it, by \cite{p3}. Non-specialty
of $\sys(9;4^{\times a},3^{\times b})$ follows from Thm. \ref{od9do13}. The computation of
$\vdim$ is straightforward:
\begin{align*}
\vdim \sys_3(9;4^{\times a},3^{\times b}) & =
\binom{13}{3}-a\binom{6}{3}-b\binom{5}{3}-1=220-10(2a+b)-1=-1,\\
\vdim \sys_3(3;2^{\times 5}) & =
\binom{6}{3} - 5 \binom{4}{3} - 1 = -1.
\end{align*}
\end{proof}

So the possible glueings are $2^{\times 5} \to 4$ and $4^{\times a},3^{\times b} \to 10$
whenever $2a+b=22$. Observe that the virtual dimension
of a system before and after glueing does not change. Hence, it is enough to consider systems
$\sys_3(d;10^{\times q},4^{\times x},3^{\times y},2^{\times z})$
with $2x+y \leq 21$ and $z \leq 4$. We summarize the above in the following
algorithm:

\medskip\noindent{\bf Algorithm B.}
{\small
\begin{verbatim}
 ----------------------------------------
 d=   -- degree should be chosen
 N=binomial(d+3,3);
 for z from 0 to 4 do
  for y from 0 to 10 do
   for x from 0 to 21 do 
(*) for q from 0 to ceilinq(N/220) do (
     if ((220*q+20*x+10*y+4*z>N-4)and(220*q+20*x+10*y+4*z<N+20)and(2*a+b<22)) then (
      -- computation of the rank of the interpolation matrix
      -- the rank is maximal if and only if the system is non-special
     )
    )
 ----------------------------------------
\end{verbatim}
}

\medskip

The above algorithm works properly for $d \geq 22$. For lower values of
$d$ the $10$-points forces systems to be special and our approach does not work.
To avoid this, we can change one line in Algorithm B
{\small
\begin{verbatim}
(*)   for q from 1 to 1 do (  -- 1 can be change into arbitrary non-negative integer
\end{verbatim}
}
\noindent
choosing the number of $10$-points to be considered. It is reasonable to
choose this number as big as possible to obtain less cases. The right choices
are one $10$-point for $d \in \{13,\dots,18\}$, five $10$-points
for $d=19$, seven for $d=20$ and eight for $d=21$.

The above algorithm has been implemented in FreePascal (but it is easy to
implement it in Singular or Macaulay2 or any other computer algebra program)
and ran on several computers simultanously. All systems appeared to
be non-special.

\begin{remark}
By Thm. 1 (``splitting theorem'') from \cite{Dum} we can avoid some of harmfull 
computations. As an
example consider $L=\sys_3(40;10^{\times 56},4^{\times 2})$. This system
appears in a list of systems to be checked for $d=40$. By 
splitting, it is enough to check non-specialty of
$\sys_3(39;10^{\times 52},4^{\times 2})$ and $\sys_3(40;40,10^{\times 4})$.
The first system is assumed to be done during the previous stage (for $d=39$),
the second is non-special by \cite{p3}. Another possibility to make
computations faster is to assume that four points are fundamental ones
($(1:0:0:0)$, $(0:1:0:0)$, $(0:0:1:0)$ and $(0:0:0:1)$). This is possible
by using a linear automorphism of $\mathbb P^3$. But then the size
of the matrix is smaller by $\sum_{j=1}^{4} \binom{m_j+2}{3}$,
so, in most cases, where at least four $10$-points are present, the
advantage is $880$ rows and columns.
For $d$ large enough we can also assume the multiplicities $m_1,\dots,m_4$
to be $14$ (resp. $15$, $18$, $20$),
since the system $\sys_3(13;4^{\times a},3^{\times b})$ for $2a+b=56$
(resp. $\sys_3(14;4^{\times a},3^{\times b})$ for $2a+b=68$,
$\sys_3(17;4^{\times a},3^{\times b})$ for $2a+b=114$,
$\sys_3(19;4^{\times a},3^{\times b})$ for $2a+b=154$)
is non-special of virtual dimension equal to $-1$.
\end{remark}

\end{document}